\documentclass{article}
\usepackage{amssymb}
\usepackage{amsthm}
\usepackage{amsmath}
\usepackage{graphicx}
\usepackage{tikz}
\usetikzlibrary{cd}
\usepackage{hyperref}
\usepackage{mathpartir}

\newtheorem{theorem}{Theorem}[section]
\newtheorem{lemma}[theorem]{Lemma}
\newtheorem{proposition}[theorem]{Proposition}
\newtheorem{corollary}[theorem]{Corollary}
\theoremstyle{definition}
\newtheorem{definition}[theorem]{Definition}
\newtheorem{assumption}[theorem]{Assumption}
\theoremstyle{remark}
\newtheorem{example}[theorem]{Example}
\newcommand{\itpr}[1]{[\![#1]\!]} 
\newcommand{\cttmodel}[1]{\widetilde{#1}} 
\newcommand{\trunc}[1]{\left\|#1\right\|} 
\newcommand{\I}{\mathbb{I}}
\newcommand{\N}{\mathbb{N}} 
\newcommand{\Izero}{0}
\newcommand{\Ione}{1}
\newcommand{\Imin}{\sqcap}
\newcommand{\Imax}{\sqcup}
\newcommand{\loc}{\mathcal{L}}  
\newcommand{\Initial}{\mathbf{0}} 
\newcommand{\Id}{\mathsf{Id}} 
\newcommand{\Path}{\mathsf{Path}} 
\newcommand{\Jop}{\mathcal{J}} 
\newcommand{\Kop}{\mathcal{K}} 
\newcommand{\ext}{\mathsf{ext}} 
\newcommand{\isext}{\mathsf{isext}} 
\newcommand{\inl}{\mathsf{inl}} 
\newcommand{\inr}{\mathsf{inr}} 
\newcommand{\Bool}{\mathbf{2}}  
\newcommand{\PSh}{\mathcal{P}}  
\newcommand{\op}{\mathrm{op}}   

\newcommand{\Asm}{\mathbf{Asm}}
\newcommand{\CAsm}{\mathbf{CAsm}} 

\newcommand{\cat}[1]{\mathcal{#1}} 
\newcommand{\Set}{\mathbf{Set}}    
\newcommand{\Ctx}{\mathbb{C}}  
\newcommand{\Type}{\mathbb{T}} 
\newcommand{\Elem}{\mathbb{E}} 
\newcommand{\Prop}{\mathbb{F}} 
\newcommand{\judg}{\mathcal{J}} 
\newcommand{\univ}{\mathcal{U}} 
\newcommand{\To}{\Rightarrow}
\newcommand{\FromTo}{\Leftrightarrow}

\title{On Church's Thesis in Cubical Assemblies}
\author{Andrew W Swan and Taichi Uemura}
\begin{document}
\maketitle

\begin{abstract}
  We show that Church's thesis, the axiom stating that all functions
  on the naturals are computable, does not hold in the cubical
  assemblies model of cubical type theory.

  We show that nevertheless Church's thesis is consistent with
  univalent type theory by constructing a reflective subuniverse of
  cubical assemblies where it holds.
\end{abstract}

\section{Introduction}
\label{sec:introduction}

One of the main branches of constructive mathematics is that of
recursive or ``Russian'' constructivism, where to justify the
existence of mathematical objects, one must show how to compute them.
A rather extreme interpretation of this
philosophy is the axiom of Church's thesis, which states that all
functions from \(\N\) to \(\N\) are computable. Despite (or perhaps
because) of its highly non-classical nature it has been well studied
by logicians and turns out to be consistent with a wide variety of
formal theories for constructive mathematics. This is usually proved
using realizability models based on computable functions, starting
with Kleene's model of Heyting arithmetic\cite{kleene1945}, but with
many later variants and generalisations. See for example \cite[Chapter
4, Section 4]{troelstravandalen} for a standard reference.

When interpreting Church's thesis in type theory an additional
complication is introduced. Logical statements are usually interpreted
in type theory using the propositions-as-types
interpretation. Applying this to Church's thesis would give us the
type below.
\[
  \prod_{f : \N \to \N}\sum_{e : \N}\prod_{x : \N}\sum_{z : \N}T(e, x, z) \land U(z)
  = f(x)
\]
However it is straightforward to use function extensionality to show
that this type is empty.\footnote{Surprisingly, however,
  this version is consistent with intensional type theory as
  long as one drops the \(\xi\) rule, which was proved by Ishihara,
  Maietti, Maschio and Streicher in
  \cite{ishiharamaiettimaschiostreicher}. They leave the
  case where one has the \(\xi\) rule but not function
  extensionality open and to the authors' knowledge it remains an
  open problem.}
It is therefore impossible in any case to show that
the above ``untruncated'' version of Church's thesis is consistent
with univalence, since univalence implies function extensionality
\cite[Theorem 4.9.4]{hottbook}.

To have any hope of showing Church's thesis is consistent with
univalence we need a different formulation. We will use the
interpretation of logical statements advocated in \cite[Section
3.7]{hottbook}, and commonly used in homotopy type theory and
elsewhere. In this
approach one uses the higher inductive type of propositional
truncation at disjunction and existential quantifiers, which ensures that
the resulting type is always an hproposition (i.e. that any two of its
elements are equal). This yields the following version of Church's
thesis, which is the one we will study here.
\[
  \prod_{f : \N \to \N}\trunc{\sum_{e : \N}\prod_{x : \N}\sum_{z : \N}T(e, x, z) \times U(z) =
  f(x)}
\]

It is well known that Church's thesis holds in the internal logic of
Hyland's effective topos (see for instance \cite[Section 3.1]{vanoosten} for a
standard reference). Similar arguments show that in fact it already
holds in its simpler subcategory of assemblies, and even in cubical
assemblies, when they are viewed as regular locally cartesian closed
categories and thereby, following Awodey and Bauer in
\cite{awodeybauer} or Maietti in \cite{maietti} as models of
extensional type theory with propositional truncation. However, the
interpretation of cubical type theory in cubical assemblies due to the
second author \cite{cubicalassemblies} is very different to the
interpretation of extensional
type theory. We draw attention in particular to the fact that for
extensional type theory hpropositions are implemented as maps where in
the internal logic each fibre has at most one
element.\footnote{Equivalently, the map is a monomorphism, and for
  this reason hpropositions are sometimes referred to as \emph{mono
    types}.}
On the other hand in the interpretation of cubical type
theory, each fibre can have multiple elements as long as any two
elements are joined by a path, telling us to always treat them as
``propositionally equal.''
For propositional truncation we don't strictly identify
elements by quotienting, but instead add new paths. Our first result
is that Church's thesis is in fact false in the interpretation of
cubical type theory in cubical assemblies, even though it holds in
the internal logic.

To show Church's thesis is consistent with univalence we will combine
cubical assemblies with the work of Rijke, Shulman and Spitters on
modalities and \(\Sigma\)-closed reflective subuniverses in
\cite{rijkeshulmanspitters}. We will construct a reflective
subuniverse where Church's thesis is forced to hold, and then use
properties of cubical assemblies to show that this reflective
subuniverse is non trivial. Our model can also be viewed as a kind of
stack model akin to those used by Coquand for various independence and
consistency results, including the independence of countable choice
from homotopy type theory \cite{coquandcubicalstacks}, although our
formulation will be quite different to Coquand's.

\subsection*{Acknowledgements}
\label{sec:acknowledgements}

The first author is grateful for some helpful discussions on higher
inductive types with Simon Huber, Anders M\"{o}rtberg and Christian
Sattler at the Hausdorff Research Institute for Mathematics during the
trimester program Types, Sets and Constructions. The second author is
supported by the research programme ``The Computational Content of
Homotopy Type Theory'' with project number 613.001.602, which is
financed by the Netherlands Organisation for Scientific Research
(NWO). We are grateful to Benno van den Berg for helpful comments and
corrections.

\section{Models of Type Theories}
\label{sec:models-type-theories}

In this paper we use models of different kinds of type theory:
extensional dependent type theory; intensional dependent type theory
(with the univalence axiom). All of them are based on the notion of a
category with families \cite{dybjer1996internal}.

\begin{definition}
  Let \(\cat{C}\) be a category. A \emph{cwf-structure over
    \(\cat{C}\)} is a pair \((T, E)\) of presheaves \(T :
  \cat{C}^{\op} \to \Set\) and \(E :
  \left(\int_{\cat{C}}T\right)^{\op} \to \Set\) such that, for any
  object \(\Gamma \in \cat{C}\) and element \(X \in T(\Gamma)\), the
  presheaf
  \[
    \left(\cat{C}/\Gamma\right)^{\op} \ni (f : \Delta \to \Gamma)
    \mapsto E(\Delta, X \cdot f) \in \Set
  \]
  is representable. The representing object for this presheaf is
  denoted by \(\chi(X) : \{X\} \to \Gamma\) or
  \(\chi(X) : \Gamma.X \to \Gamma\). A \emph{category with families},
  cwf in short, is a triple
  \(\cat{E} = (\Ctx^{\cat{E}}, \Type^{\cat{E}}, \Elem^{\cat{E}})\)
  such that \(\Ctx^{\cat{E}}\) is a category with a terminal object
  and \((\Type^{\cat{E}}, \Elem^{\cat{E}})\) is a cwf-structure over
  \(\Ctx^{\cat{E}}\).
\end{definition}

In general a model \(\cat{E}\) of a type theory consists of a category
with families \((\Ctx^{\cat{E}}, \Type^{\cat{E}}, \Elem^{\cat{E}})\)
and algebraic operations on the presheaves \(\Type^{\cat{E}}\) and
\(\Elem^{\cat{E}}\). An object \(\Gamma\) of \(\Ctx^{\cat{E}}\) is
called a \emph{context}. An element \(X\) of
\(\Type^{\cat{E}}(\Gamma)\) is called a \emph{type} and written
\(\Gamma \vdash_{\cat{E}} X\). An element \(a\) of
\(\Elem^{\cat{E}}(\Gamma, X)\) is called an \emph{element of type
  \(X\)} and written \(\Gamma \vdash_{\cat{E}} a : X\). The subscript
of \(\vdash_{\cat{E}}\) is omitted when the model \(\cat{E}\) is clear
from the context. An algebraic operation on those presheaves is
expressed by the schema
\[
  \inferrule
  {\Gamma \vdash \judg_{1} \\
    \dots \\
    \Gamma \vdash \judg_{n}}
  {\Gamma \vdash A(\judg_{1}, \dots, \judg_{n})}
\]
where \(\Gamma \vdash \judg_{j}\) and
\(\Gamma \vdash A(\judg_{1}, \dots, \judg_{n})\) are either of the
form \(\Gamma.X_{1}.\dots.X_{m} \vdash Y\) or of the form
\(\Gamma.X_{1}.\dots.X_{m} \vdash b : Y\) with
\((\Gamma \vdash X_{1}), \dots, (\Gamma.X_{1}.\dots.X_{m-1} \vdash
X_{m})\). In this schema we always assume that the operation
\(A(\judg_{1}, \dots, \judg_{n})\) is stable under reindexing: if
\(f : \Delta \to \Gamma\) is a morphism in \(\Ctx^{\cat{E}}\), then we
have
\(A(\judg_{1}, \dots, \judg_{n}) \cdot f = A(\judg_{1} \cdot f, \dots,
\judg_{n} \cdot f)\).

\begin{example}
  Let \(\cat{E}\) be a cwf. We say \(\cat{E}\) supports
  \emph{dependent product types} if it has operations
  \begin{mathpar}
    \inferrule
    {\Gamma \vdash X \\
      \Gamma.X \vdash Y}
    {\Gamma \vdash \Pi(X, Y)}
    \and
    \inferrule
    {\Gamma \vdash X \\
      \Gamma.X \vdash Y \\
      \Gamma.X \vdash b : Y}
    {\Gamma \vdash \lambda(X, Y, b) : \Pi(X, Y)}
  \end{mathpar}
  such that the map \(\Elem(\Gamma.X, Y) \ni b \mapsto \lambda(X, Y,
  b) \in \Elem(\Gamma, \Pi(X, Y))\) is bijective.
\end{example}

It is a kind of routine to describe other type constructors such as
dependent sum types, extensional and intensional identity types,
inductive types, higher inductive types and universes.

For a model \(\cat{E}\) of a type theory, we denote by
\(\itpr{-}^{\cat{E}}\) the interpretation of the type theory in the
model \(\cat{E}\).

\begin{definition}
  By a \emph{model of univalent type theory} we mean a cwf that
  supports dependent product types, dependent sum type, intensional
  identity types, unit type, finite coproducts, natural numbers,
  propositional truncation and a countable chain
  \[
    \univ_{0} : \univ_{1} : \univ_{2} : \dots
  \]
  of univalent universes.
\end{definition}

\subsection{Internal Languages}
\label{sec:internal-languages}

Formally we will work with models of type theories, but we will
construct types and terms of those models in a syntactic way using
their internal languages. Let
\(\cat{E} = (\Ctx^{\cat{E}}, \Type^{\cat{E}}, \Elem^{\cat{E}},
\dots)\) be a model of a type theory. For a context
\(\Gamma \in \Ctx^{\cat{E}}\) and a type \(\Gamma \vdash X\),
we introduce a variable \(x\) and write \((\Gamma, x : X)\) for the
context \(\Gamma.X\). For another type \(\Gamma \vdash Y\),
the weakening \(\Gamma, x : X \vdash Y\) is interpreted as the
reindexing \(\Gamma.X \vdash Y \cdot \chi(X)\). For an element
\(\Gamma \vdash a : X\) and a type
\(\Gamma, x : X \vdash Y(x)\), the substitution
\(\Gamma \vdash Y(a)\) is interpreted as the reindexing
\(\Gamma \vdash Y \cdot \bar{a}\), where
\(\bar{a} : \Gamma \to \Gamma.X\) is the section of
\(\Gamma.X \to \Gamma\) corresponding to the element
\(\Gamma \vdash a : X\). All type and term constructors of the type
theory are soundly interpreted in \(\cat{E}\) in a natural way. Note
that types and terms built in the internal language are stable under
reindexing.

\subsection{\(W\)-types with Reductions}
\label{sec:w-types-with-1}

We will later use \(W\)-types with reductions to construct higher
inductive types. So that we can use them internally in type theory we
give below a new, split formulation. This is based on the
non-dependent special case of the version in
\cite{swanwtypesreductions}.

Let
\(\cat{E} = (\Ctx^{\cat{E}}, \Type^{\cat{E}}, \Elem^{\cat{E}},
\dots)\) be a model of a type theory with dependent product types,
dependent sum types and extensional identity types. Suppose that
\(\cat{E}\) has types \(1 \vdash \Prop\) and
\(\varphi : \Prop \vdash [\varphi]\) such that
\(\varphi : \Prop, x : [\varphi], y : [\varphi] \vdash x = y\). We
call an element of \(\Prop\) a cofibrant proposition. We often omit
\([-]\) and regard an element \(\varphi : \Prop\) itself as a type. A
\emph{cofibrant polynomial with reductions} over a context
\(\Gamma \in \Ctx^{\cat{E}}\) consists of the following data:
\begin{itemize}
\item a type \(\Gamma \vdash Y\) of \emph{constructors};
\item a type \(\Gamma, y : Y \vdash X(y)\) of \emph{arities};
\item a cofibrant proposition \(\Gamma, y : Y \vdash R(y) : \Prop\)
  together with an element \(\Gamma, y : Y, r : R(y) \vdash k(y, r) :
  X(y)\) which we refer to as the \emph{reductions}.
\end{itemize}
An \emph{algebra} for a cofibrant polynomial with reductions
\((Y, X, R, k)\) over \(\Gamma \in \Ctx^{\cat{E}}\) is a type
\(\Gamma \vdash W\) together with an element
\(\Gamma, y : Y, \alpha : X(y) \to W \vdash s(y, \alpha) : W\) such
that
\(\Gamma, y : Y, \alpha : X(y) \to W, r : R(y) \vdash s(y, \alpha) =
\alpha(k(y, r))\). Algebras for \((Y, X, R, k)\) form a category in
the obvious way and we say \(\cat{E}\) supports \emph{cofibrant
  \(W\)-types with reductions} if every cofibrant polynomial with
reductions has an initial algebra preserved by reindexing.

\section{Orton-Pitts Construction}
\label{sec:orton-pitts}
\begingroup
\newcommand{\E}{\mathcal{E}}
\newcommand{\id}{\mathrm{id}}

\begin{assumption}
  \label{asm:orton-pitts}
  Let \(\cat{E}\) be a model of dependent type theory that supports
  dependent product types, dependent sum types, extensional identity
  types, unit type, finite colimits, natural numbers, propositional
  truncation and a countable chain of universes. We further assume
  that every context \(\Gamma \in \Ctx^{\cat{E}}\) is isomorphic to
  \(1.X\) for some type \(X\) over the terminal object \(1\). Suppose
  the following:
  \begin{itemize}
  \item \(\cat{E}\) has a type \(1 \vdash \I\) equipped with
    two constants \(\Izero\) and \(\Ione\) and two binary operators
    \(\Imin\) and \(\Imax\);
  \item \(\cat{E}\) has types \(1 \vdash \Prop\) and
    \(\varphi : \Prop \vdash [\varphi]\) such that
    \(\varphi : \Prop, x : [\varphi], y : [\varphi] \vdash x = y\). An
    element of \(\Prop\) is called a cofibrant proposition. We often
    omit \([-]\) and regard an element \(\varphi : \Prop\) itself as a
    type;
  \item \(\I\) and \(\Prop\) satisfy
    \(\mathtt{ax_{1}}\)--\(\mathtt{ax_{9}}\) given by Orton and Pitts
    \cite{orton2018axioms};
  \item \(\Prop\) satisfies propositional extensionality:
    \(\prod_{\varphi, \psi : \Prop}(\varphi \FromTo \psi) \To (\varphi
    = \psi)\);
  \item the exponential functor
    \((-)^{\I} : \Ctx^{\cat{E}} \to \Ctx^{\cat{E}}\) has a right
    adjoint;
  \item \(\cat{E}\) supports cofibrant \(W\)-types with reductions.
  \end{itemize}
\end{assumption}

Note that the axioms in \cite{orton2018axioms} are written in the
internal language of an elementary topos, but they are easily
translated into dependent type theory with \(\I\) and \(\Prop\) as
above. We require propositional extensionality which trivially
holds when \(\Prop\) is a subobject of the subobject classifier of an
elementary topos.

Under these assumptions, we will build a model \(\cttmodel{\cat{E}}\)
of univalent type theory as follows:
\begin{itemize}
\item the base category \(\Ctx^{\cttmodel{\cat{E}}}\) is that of
  \(\cat{E}\);
\item the types \(\Gamma \vdash_{\cttmodel{\cat{E}}} X\) are
  the types \(\Gamma \vdash_{\cat{E}} X\) equipped with a
  ``fibration structure'';
\item the elements \(\Gamma \vdash_{\cttmodel{\cat{E}}} a : X\) are
  the elements \(\Gamma \vdash_{\cat{E}} a : X\) of the underlying
  type \(X\) in \(\cat{E}\);
\end{itemize}
By the construction given in \cite{orton2018axioms}, this model
\(\cttmodel{\cat{E}}\) supports dependent product types, dependent sum
types, identity types, unit type, finite coproducts and natural
numbers. For a countable chain of univalent universes, use the right
adjoint to \((-)^{\I}\) as in \cite{licata2018internal}. It remains to
show that \(\cttmodel{\cat{E}}\) supports propositional truncation,
which will be proved in Section \ref{sec:high-induct-types} using
cofibrant \(W\)-types with reductions. We call a model of univalent
type theory of the form \(\cttmodel{\cat{E}}\) an \emph{Orton-Pitts
  model}.

\subsection{Higher Inductive Types in Orton-Pitts Models}
\label{sec:high-induct-types}
\begingroup
\newcommand{\hcomp}{\mathsf{hcomp}}
\newcommand{\cone}{\mathsf{Cone}}
\newcommand{\susp}{\mathsf{Susp}}
\newcommand{\presusp}{\mathsf{Susp}_0}
\newcommand{\paste}{\mathsf{pastecone}}
\newcommand{\cof}{\Prop}
\newcommand{\lfr}{\mathsf{LFR}}
\newcommand{\incl}{\mathsf{inc}}
\newcommand{\sq}{\mathsf{sq}}

We are still working with a model \(\cat{E}\) of type theory that
satisfies Assumption \ref{asm:orton-pitts}. We will show how to
construct higher inductive types in \(\cttmodel{\cat{E}}\). Our
techniques are fairly general, although we will focus on the HITs that
we will need for the main theorem. The techniques developed by Coquand
Huber and M\"{o}rtberg in \cite{coquandhubermortberg} are already very
close to working in arbitrary Orton-Pitts models. The only exception
is that the underlying objects for the HITs are given by certain
initial algebras, which are constructed directly for cubical
sets. This definition doesn't quite work for cubical assemblies for
two reasons. Firstly we are using a different cube category, and
secondly we are working internally in assemblies. Rather than proving
the same results again for cubical assemblies we will use a more
general approach based on \(W\)-types with reductions that covers both
cases.  The first author already showed in \cite[Section
4]{swanwtypesreductions} that (non-split) \(W\)-types with locally
decidable reductions can be constructed in any category of presheaf
assemblies and we'll see later how to ensure that we get in fact
\emph{split} \(W\)-types with reductions in presheaf assemblies.

Finally, we will also make some minor adjustments related to the fact
that we do not assume the interval object has reversals.

When we construct higher inductive types, we will use formulations
based on \(\Path\) types, following Coquand, Huber and
M\"{o}rtberg. Technically these formulations can only be stated in
cubical type theory, and not in intensional type theory in
general. However, it is straightforward to derive versions based on
\(\Id\) types using the equivalence of \(\Path\) and \(\Id\) types,
which are then valid in \(\cttmodel{\cat{E}}\). We note that although
computation rules hold definitionally for both point and path
constructors for the \(\Path\) type versions, after translating to
\(\Id\) types, the definitional equality only holds for point
constructors. However, neither definitional equality will be needed
for our end result.

\begin{definition}
  Given a type \(\Gamma \vdash_{\cat{E}} A\), we define the
  \emph{local fibrant replacement of \(A\)}, \(\lfr(A)\) to be the
  \(W\)-type with reductions defined as follows.
  \begin{itemize}
  \item When \(a : A\), we add an element \(\incl(a)\) to
    \(\lfr(A)\).
  \item When \(\varphi : \cof\), \(\epsilon \in \{0, 1\}\) and
    \(u : \sum_{i : \I} ((i = \epsilon) \vee \varphi) \;\to\; \lfr(A)\), we
    add an element \(\hcomp(\varphi, \epsilon, u)\) to \(\lfr(A)\).
  \item If \(p : \varphi\) and \(\epsilon\) and \(u\) are as above
    then \(\hcomp(\varphi, \epsilon, u)\) reduces to
    \(u(1 - \epsilon, p)\).
  \end{itemize}
  Formally, we define the constructors \(Y\) to be the coproduct
  \(A + (\cof \times 2)\). We take the arity \(X(\inl(a))\) to be the
  empty type for \(a : A\) and \(X(\varphi, \epsilon)\) to be
  \(\sum_{i : \I} \varphi \vee (i = \epsilon)\) for
  \((\varphi, \epsilon) : \cof \times 2\). We take the reductions
  \(R(\inl(a))\) to be \(\bot\) for \(a : A\) and
  \(R(\varphi, \epsilon)\) to be \(\varphi\) together with the map
  \(p : \varphi \vdash (1 - \epsilon, p) : \sum_{i : \I} \varphi \vee
  (i = \epsilon)\).
\end{definition}

\begin{theorem}
  \label{thm:suspensionexist}
  The model \(\cttmodel{\cat{E}}\) supports suspensions.
\end{theorem}
\begin{proof}
  Suppose we are given a type
  \(\Gamma \vdash_{\cttmodel{\cat{E}}} X\). We first construct the
  na\"{i}ve suspension, \(\presusp(X)\) as the pushout below.
  \begin{equation*}
    \begin{tikzcd}
      X \times 2 \arrow[r] \arrow[d] & 2 \arrow[d] \\
      X \times \I \arrow[r] & \ar[ul, phantom, "\ulcorner", very near
      start]
      \presusp(X)
    \end{tikzcd}
  \end{equation*}

  We next take the local fibrant replacement, to get
  \(\lfr(\presusp(X))\). This is then an initial \(\susp(X)\) algebra,
  as defined by Coquand, Huber and M\"{o}rtberg in \cite[Section
  2.2]{coquandhubermortberg} and so we can then proceed with the same
  proof as they do there.
\end{proof}

\begin{theorem}
  The model \(\cttmodel{\cat{E}}\) supports propositional truncation.
\end{theorem}
\begin{proof}
  Suppose we are given a type \(\Gamma
  \vdash_{\cttmodel{\cat{E}}}A\). We first define the underlying object
  of \(\| A \|\) to be the \(W\)-type with reductions defined as
  follows.
  \begin{itemize}
  \item When \(a : A\), we add an element \(\incl(a)\) to
    \(\| A \|\).
  \item When \(\varphi : \cof\), \(\epsilon \in \{0, 1\}\) and
    \(u : \sum_{i : \I} ((i = \epsilon) \vee \varphi) \;\to\; \|A \| \), we
    add an element \(\hcomp(\varphi, \epsilon, u)\) to \(\| A \|\).
  \item If \(p : \varphi\) and \(\epsilon\) and \(u\) are as above then
    \(\hcomp(\varphi, \epsilon, u)\) reduces to \(u(1-\epsilon, p)\).
  \item If \(x, y : \| A \|\) and \(i : \I\), then \(\| A \|\)
    contains an element of the form \(\sq(x, y, i)\).
  \item If \(x, y, i\) are as above and \(i = 0\), then \(\sq(x, y,
    i)\) reduces to \(x\).
  \item If \(x, y, i\) are as above and \(i = 1\), then \(\sq(x, y,
    i)\) reduces to \(y\).
  \end{itemize}
  Formally, we define this by taking the coproduct of two polynomials
  with reductions. The first is the one we used before for
  \(\lfr\). The second has constructors \(Y := \I\), with the arity
  defined by \(X(i) := 2\), and reductions \(R(i) := (i = 0) \lor (i =
  1)\) together with the map \(p : (i = 0) \lor (i = 1) \vdash k(p) :
  2\) defined by \(k(p) = 0\) if \(p : i = 0\) and \(k(p) = 1\) if \(p
  : i = 1\).

  The remainder of the proof is the same as the syntactic description
  of propositional truncation by Coquand, Huber
  and M\"{o}rtberg in \cite[Section 3.3.4]{coquandhubermortberg}.
\end{proof}

We now construct a new higher inductive type, which is a simplified
version of the higher inductive type \(\Jop_F\) defined by Rijke,
Shulman and Spitters in \cite[Section 2.2]{rijkeshulmanspitters}.
Given families of types \(\Gamma \vdash_{\cttmodel{\cat{E}}} A\) and
\(\Gamma, a : A \vdash_{\cttmodel{\cat{E}}} B(a)\) we will construct a
higher inductive type \(\Kop^\Gamma_B\) defined as follows.
\begin{itemize}
\item When \(a : A\) and \(f : B(a) \to \Kop_B\), we add an element
  \(\ext(a, f)\) to \(\Kop_B\).
\item When \(a : A\), \(f : B(a) \to \Kop_B\) and \(b : B(a)\) we add
  an element \(\isext(a, f, b)\) to \(\Path(\ext(a, f), f(b))\).
\end{itemize}
We require that \(\Kop_B\) satisfies the following elimination rule.
Suppose we are given a family of types
\(\Gamma, x : \Kop_B \vdash_{\cttmodel{\cat{E}}} P(x)\) together with
the terms below.
\begin{align*}
  R &: \prod_{a : A} \prod_{f : B(a) \to \Kop_B} \left(\prod_{b : B(a)}
      P(f(b))\right) \to P(\ext(f, c)) \\
  S &: \prod_{a : A} \prod_{f : B(a) \to \Kop_B} \,\prod_{f' : \prod_{b :
      B(a)} P(f(b))} \,\prod_{b : B(a)} \prod_{i : \I} P(\isext(a, f, b)(i))
\end{align*}
Suppose further that \(S\) satisfies the equalities
\begin{align*}
  S(a, f, f', b, 0) &= R(f') \\
  S(a, f, f', b, 1) &= f'(b)
\end{align*}
Then we have a choice of term
\(\Gamma, x : \Kop_B \vdash s(x) : P(x)\) satisfying the following
computation rules for \(a : A\), \(f : B(a) \to \Kop_B\) and \(b :
B(a)\).
\begin{align*}
  s(\ext(a, f)) &= R(a, f, s \circ f) \\
  s(\isext(a, f, b)(i)) &= S(a, f, s \circ f, b, i)
\end{align*}
Moreover the choice of term is strictly preserved by reindexing.

We use the techniques developed by Coquand, Huber and M\"ortberg
together with \(W\)-types with reductions for constructing the
actual objects. In order to give \(\Kop^\Gamma_B\) the structure of a
fibration we need to define a composition operator. We will do this by
freely adding an \(\hcomp\) operator, and then combining it with a
transport operator, which we will explicitly define.

\begin{definition}
  Let \(\Gamma \vdash_{\cat{E}} X\) be a type. We define the
  \emph{na\"ive cone}, \(\cone(X)\) to be the following
  pushout\footnote{Pushout in the usual categorical sense, not the
    homotopy pushout.}.
  \begin{equation*}
    \begin{tikzcd}
      X \arrow[r] \arrow[d, "{\langle 1_X, \delta_0 \rangle}"] & 1
      \arrow[d, "\inl"] \\
      X \times \I \arrow[r, "\inr"] & \cone(X)
    \end{tikzcd}
  \end{equation*}
\end{definition}

We can now define \(\Kop^\Gamma_B\) to be the following \(W\)-type
with reductions.
\begin{itemize}
\item When \(a : A\), \(c : \cone(B(a))\) and \(f : B(a) \to \Kop_B\),
  we add an element \(\paste(a, c, f)\) to \(\Kop_B\).
\item If \(a, c, f\) are as above and \(c\) is of the form
  \(\inr(b, 1)\) for \(b : B(a)\), then \(\paste(a, c, f)\) reduces to
  \(f(b)\).
\item When \(\varphi : \cof\) and
  \(u : \sum_{i : \I} ((i = 0) \vee \varphi) \;\to\; \Kop_B\), we
  add an element \(\hcomp(\varphi, u)\) to \(\Kop_B\).
\item If \(p : \varphi\) and \(u\) is as above then
  \(\hcomp(\varphi, u)\) reduces to \(u(1, p)\).
\end{itemize}

To check that this really is a \(W\)-type with reductions, we
need to define the polynomial with reductions. We take it to be
the coproduct of the following two polynomials with reductions.

We define the first component of the coproduct as follows. We take the
constructors \(Y\) to be \(\sum_{a : A} \cone(B(a))\) and the
arities \(X(a, c)\) to be \(B(a)\). We take the reductions
\(R(a, \inl(\ast))\) to be \(\bot\) and \(R(a, \inr(b, i))\) to be
\((i = 1)\) together with the map
\(R(a, \inr(b, i)) \vdash b : B(a)\). Note that
\(R : \left(\sum_{a : A}\cone(B(a))\right) \to \Prop\) is
well-defined because we have \((0 = 1) = \bot\) by propositional
extensionality.

The second component in the coproduct is the polynomial with
reductions that we used for local fibrant replacement.

\begin{lemma}
  \label{lem:koptransport}
  We construct a transport operator for \(\Kop^\Gamma_B\), in the
  sense defined in \cite[Definition 2.3]{coquandhubermortberg}.
\end{lemma}
\begin{proof}
  Suppose we are given \(\varphi : \cof\) and a path \(\gamma\) in
  \(\Gamma\) which is constant on \(\varphi\). We need to define a
  transport operator, which is a map
  \(t : \Kop_{B(\gamma(0))} \to \Kop_{B(\gamma(1))}\) such that \(t\)
  is the identity when \(\varphi\) is true. Formally this map can be
  defined by giving an appropriate algebra structure on
  \(\Kop_{B(\gamma(1))}\) and then using the initiality of
  \(\Kop_{B(\gamma(0))}\). However, for clarity we will present the
  proof as an argument by higher recursion on the definition of
  \(\Kop_{B(\gamma(0))}\).

  We need to show how to define
  \(t(\paste(a, c, f))\) and \(t(\hcomp(\psi, u))\), and then check
  that the definition respects the reduction equations. For the latter
  we define the transport operator so that it preserves the
  \(\hcomp\) structure, which determines it uniquely, following
  \cite{coquandhubermortberg}. For the
  former, we recall that \(\cone(B(a))\) was defined as a pushout, and
  so we can split into a further two cases. Either \(c\) is of the
  form \(\inl(\ast)\), or it is of the form \(\inr(b, i)\) where \(b :
  B(a)\) and \(i : \I\). Now in addition to the reduction equation, we
  have to also satisfy \(t(\inl(\ast)) = t(\inr(b, 0))\) in order to
  eliminate out of the pushout.

  Write \(t_A\) for the transport \(A(\gamma(0)) \to A(\gamma(1))\)
  and \(t_B\) for the transport
  \(\prod_{a : A(\Gamma(0))} B(a) \to B(t_A(a)) \) ensuring that
  \(t_A(a) = a\) and \(t_B(b) = b\) when \(\varphi = \top\), for all
  \(a : A(\gamma(0))\) and \(b : B(a)\). Write
  \(t_B^{-1}\) for the homotopy inverse
  \(\prod_{a : A(\Gamma(0))} B(t_A(a)) \to B(a)\), again ensuring
  that \(t_B^{-1}(b) = b\) when \(\varphi = \top\). Since we are only
  guaranteed the existence of a homotopy inverse, not a strict
  inverse, we don't necessarily have
  \(t_B^{-1} \circ t_B = 1_{B(a)}\). We can however construct paths
  \(p : \prod_{a : A(\Gamma(0))} \prod_{b : B(a)} \I \to B(a)\) satisfying
  for all \(a : A(\Gamma(0))\) and \(b : B(a)\) that
  \(p(a, b, 0) = t_B^{-1}(t_B(b)) \) and \(p(a, b, 1) =
  b\). Furthermore, we may assume that for any \(a, b\) and \(i\), if
  \(\varphi = \top\) then \(p(a, b, i) = b\).

  We define \(t(\paste(a, \inl(\ast), f))\) to be of the form
  \(\paste(t_A(a), \inl(\ast), f')\), where we still need to define a
  function \(f' : B(t_A(a)) \to \Kop_{B(\gamma(1))}\). Note that we
  may assume by recursion that for each \(b : B(a)\), \(t(f(b))\) has
  already been defined and belongs to \(\Kop_{B(\gamma(1))}\). Hence
  we can simply define \(f'\) to be \(t \circ f \circ t_B^{-1}\).

  The obvious first attempt at defining
  \(t(\paste(a, \inr(b, i), f))\), would be
  \(\paste(t_A(a), \inr(t_B(b), i), t \circ f \circ t_B^{-1})\). Note
  however that this does not satisfy the reduction equations. This is
  because when \(i = 1\), \(\paste(a, \inr(b, i), f)\) reduces to
  \(f(b)\) and
  \(\paste(t_A(a), \inr(t_B(b), i), t \circ f \circ t_B^{-1})\)
  reduces to \(t(f(t_B^{-1}(t_B(b))))\) which is not necessarily
  strictly equal to \(t(f(b))\). We fix this using the \(\hcomp\)
  constructor, following the construction of homotopy pushouts in
  \cite[Section 2.3]{coquandhubermortberg}. We define \(\psi : \cof\)
  to be \(\varphi \vee (i = 0) \vee (i = 1)\). We then define
  \(u : \sum_{j : \I} (\psi \vee (j = 0)) \to
  \Kop_{B(\gamma(a))}\) as follows.
  \begin{equation*}
    u(j, \ast) :\equiv
    \begin{cases}      
      \paste(t_A(a), \inr(t_B(b), i), t \circ f \circ t_B^{-1})
      & j = 0 \\
      \paste(a, \inr(b, i), t \circ f) & \varphi = \top \\
      \paste(t_A(a), \inl(\ast), t \circ f \circ t_B^{-1}) & i = 0 \\
      t(f(p(a, b, j))) & i = 1
    \end{cases}
  \end{equation*}
  We then define \(t(\paste(a, \inr(b, i), f))\) to be
  \(\hcomp(\psi, 0, u)\). The reduction equation for \(\hcomp\) then
  ensures that we do satisfy the reduction equation for \(\paste\) and
  also retain the necessary equations for the pushout and furthermore
  ensures that the resulting map \(t : \Kop_{B(\gamma(0))} \to
  \Kop_{B(\gamma(1))}\) is a transport operator.
\end{proof}

\begin{theorem}
  We construct a fibration structure for each \(\Kop_B\), which is
  strictly preserved by reindexing.
\end{theorem}
\begin{proof}
  By lemma \ref{lem:koptransport} and \cite[Lemma
  2.5]{coquandhubermortberg}.
\end{proof}

\begin{lemma}
  We construct terms \(\ext\) and \(\isext\) for \(\Kop_B\) that
  satisfy the appropriate equations.
\end{lemma}
\begin{proof}
  \begin{align*}
    \ext(a, f) &:\equiv \paste(a, \inl(\ast), f) \\
    \isext(a, f, b)(i) &:\equiv \paste(a, \inr(b, i), f)
  \end{align*}
\end{proof}

\begin{lemma}
  \(\Kop_B\) satisfies the necessary induction principle.
\end{lemma}
\begin{proof}
  Suppose we are given a family of types
  \(\Gamma, x : \Kop_B \vdash_{\cttmodel{\cat{E}}} P(x)\) together
  with the terms below.
  \begin{align*}
    R &: \prod_{a : A} \prod_{f : B(a) \to \Kop_B} \left(\prod_{b : B(a)}
        P(f(b))\right) \to P(\ext(f, c)) \\
    S &: \prod_{a : A} \prod_{f : B(a) \to \Kop_B} \, \prod_{f' : \prod_{b :
        B(a)} P(f(b))} \, \prod_{b : B(a)} \prod_{i : \I} P(\isext(a, f, b)(i))
  \end{align*}

  We need to define a term \(\Gamma, x : \Kop_B \vdash s(x) : P(x)\)
  satisfying the appropriate equalities. We define \(s\) by higher
  recursion on the construction of \(\Kop_B\). We first deal with the
  case \(s(\paste(a, c, f))\). Recalling that \(\cone(B(a))\) is
  defined as a pushout, we can split into the two cases
  \(c = \inl(\ast)\) and \(c = \inr(b, i)\) for some \(b : B(a)\) and
  \(i : \I\).

  We define
  \begin{align*}
    s(\paste(a, \inl(\ast), f)) &:\equiv R(a, f, s \circ f) \\
    s(\paste(a, \inr(b, i), f)) &:\equiv S(a, f, s \circ f, b, i)
  \end{align*}
  It is straightforward to check that this does preserve the reduction
  and pushout equations and so does give a well defined map. One can
  show it is a section again by higher recursion and the computation
  rules are satisfied by definition.

  Finally, to define \(s(\hcomp(\varphi, u))\) we use the fibration
  structure on \(\Gamma, x : \Kop_B \vdash_{\cat{E}} P(x)\).
\end{proof}

\endgroup
\endgroup

\subsection{Internal Cubical Models}
\label{sec:presheaf-models}
\begingroup
\newcommand{\Sx}{\mathcal{S}}
\newcommand{\B}{\Box}

Let \(\Sx\) be a model of dependent type theory with dependent product
types, dependent sum types, extensional identity types, unit type,
finite colimits, \(W\)-types and a countable chain of universes. We
also assume that every context of \(\Sx\) is isomorphic to \(1.X\) for
some type \(1 \vdash_{\Sx} X\). In particular, the category
\(\Ctx^{\Sx}\) is finitely complete so that internal categories in
\(\Ctx^{\Sx}\) make sense. Let \(\B\) denote the internal category in
\(\Ctx^{\Sx}\) in which the objects are the natural numbers and the
morphisms from \(n\) to \(m\) are the order-preserving functions
\(\Bool^{n} \to \Bool^{m}\). Note that \(\Sx\) has a natural number
object since it has \(W\)-types. We will refer to internal presheaves
over \(\B\) as \emph{internal cubical objects}.

\begin{theorem}
  Under those assumptions, the category of internal cubical objects in
  \(\Sx\) is part of a model of type theory that satisfies Assumption
  \ref{asm:orton-pitts}.
\end{theorem}

\begingroup
\newcommand{\pca}{\mathcal{A}}
\newcommand{\UFam}{\mathbf{UFam}}

\begin{example}
  Let \(\pca\) be a partial combinatory algebra. It is well-known that
  the category \(\Asm(\pca)\) of assemblies on \(\pca\) is part of a
  model of type theory with dependent product types, dependent sum
  types, extensional identity types, unit type, finite colimits. It is
  also known that \(\Asm(\pca)\) has \(W\)-types (an explicit
  construction is found in \cite[Section
  2.2]{vandenbergthesis}). Assuming a countable chain of Grothendieck
  universes in the set theory, \(\Asm(\pca)\) has a countable chain of
  universes. Thus the category \(\CAsm(\pca)\) of internal cubical
  objects in \(\Asm(\pca)\) is part of a model of type theory that
  satisfies Assumption \ref{asm:orton-pitts}.
\end{example}
\endgroup

It is shown in \cite{orton2018axioms} that, when \(\Sx = \Set\), the
category of presheaves over \(\B\) satisfies all the axioms of Orton
and Pitts if we take \(\Prop\) to be the presheaf of locally decidable
propositions. The proof works for an arbitrary \(\Sx\) and one can
show that the category of internal cubical objects in \(\Sx\) is part
of a model of type theory satisfying Assumption \ref{asm:orton-pitts}
except the existence of cofibrant \(W\)-types with reductions (see
also \cite{cubicalassemblies}). To construct cofibrant \(W\)-types
with reductions, we recall the following from
\cite{swanwtypesreductions}.

\begingroup
\newcommand{\E}{\mathcal{E}}
\newcommand{\C}{\mathbf{C}}
\newcommand{\yoneda}[1]{\C(-, #1)}

\begin{theorem}
  Let \(\E\) be a locally cartesian closed category with finite
  colimits and disjoint coproducts and \(W\)-types, and let \(\C\) be
  an internal category in \(\E\). Then the category \(\PSh(\C)\) of
  internal presheaves over \(\C\) has all locally decidable
  \(W\)-types with reductions.
\end{theorem}

We furthermore observe that one can show that this construction is
stable under pullback up to isomorphism using a technique similar to
the one used by Gambino and Hyland for ordinary \(W\)-types. The
reason is that pointed polynomial endofunctors are stable under
pullback because they are constructed from \(\Sigma\) types, \(\Pi\)
types and pushouts, all of which are preserved by pullback, and in
locally cartesian closed categories the initial algebras of such
pointed endofunctors are also stable under pullback. However, to
ensure that the construction is strictly preserved requires a little
more work.

We show how to use the non split version above to construct split
\(W\)-types with reductions. The essential idea is to carry out the
construction given above ``pointwise,'' expanding out the method
suggested by Coquand, Huber and M\"{o}rtberg in \cite[Section
2.2]{coquandhubermortberg}. Since we define cubical sets here as a
category of presheaves in the usual, contravariant sense, we work with
contravariant presheaves here, although the original proof in
\cite{swanwtypesreductions} is phrased in terms of covariant
presheaves. We also make minor adjustments to fit with the ``split''
version appearing in section \ref{sec:w-types-with-1}.

Suppose that we are given a context \(\Gamma \in \PSh(\C)\) together
with a type \(Y \in \PSh(\int_\C \Gamma)\), a type
\(X \in \PSh(\int_C \{Y\})\), a locally decidable monomorphism
\(R \rightarrowtail Y\) and a map \(k : \prod_{y
  : R} X(y)\) over $\int_\C \Gamma$.

We need to show how to define a strict version of the \(W\)-type with
reductions \(W(Y, X, R)\). We will refer to the new strict version as
\(W'(Y, X, R)\). This should be an element of
\(\PSh(\int_\C \Gamma)\), so in particular we need to define a family
of types \(W'(Y, X, R)(c, \gamma)\) indexed by objects \(c\) of \(\C\)
and elements \(\gamma : \Gamma(c)\).

We fix such a \(c\) and \(\gamma\). We first note that we have a a
locally decidable polynomial with reductions
\(Y_\gamma, X_\gamma, R_\gamma\) in the internal presheaf category
\(\PSh(\int_\C \yoneda{c})\) given by reindexing along the map
\(\yoneda{c} \to \Gamma\) given by Yoneda. We then carry out the ``non
strict'' construction to get a presheaf
\(W(Y_\gamma, X_\gamma, R_\gamma)\) on \(\int_\C \yoneda c\) and
finally we define \(W'(Y, X, R)(c, \gamma)\) to be
\(W(Y_\gamma, X_\gamma, R_\gamma) (c, 1_c)\).

For completeness, we unfold the definitions to obtain the following
explicit description of \(W'(Y, X, R)(c, \gamma)\). We first define
the dependent \(W\)-type \(N_0\) of \emph{normal forms} indexed by the
objects \((d, f)\) of \(\int_\C \yoneda{c}\).

If \((d, f)\) is an object of \(\int_\C \yoneda{c}\) we add an element
to \(N_0(d, f)\) of the form \(\sup(y, \alpha)\) whenever \(y\) is an
element of \(Y(d, \Gamma(f)(\gamma))\) that does not belong to the
subobject \(R(d, \Gamma(f)(\gamma))\) and
\(\alpha\) is an element of the following type.
\begin{equation*}
  \prod_{g : e \to d} N_0(e, f \circ g)^{X(f \circ g, \Gamma(f\circ
    g)(\gamma), Y(g)(y))}
\end{equation*}

The next step is to define maps \(N_0(d, f) \to N_0(e, f \circ g)\)
whenever \(g \colon e \to d\) and \(f \colon d \to c\) in \(\C\). Say
that we are given an element of \(N_0(d, f)\) of the form
\(\sup(y, \alpha)\).  We recall that \(N_0(g)(\sup(y, \alpha))\) is
defined by splitting into cases depending on whether or not \(y\)
belongs to the subobject \(R(d, \Gamma(f)(\gamma))\). If it does, we
define \(N_0(g)(\sup(y, \alpha))\) to be \(\alpha(g,
k(y))\). Otherwise, we define \(N_0(g)(\sup(y, \alpha))\) to be
\(\sup(Y(g)(y), \alpha')\) where \(\alpha'(h, x)\) is defined to be
\(\alpha(g \circ h, x)\).

We then define \(N(d, f)\) for each \(f : d \to c\) to be the
subobject of \(N_0(d, f)\) consisting of hereditarily natural elements
and verify that this does indeed define a presheaf on
\(\int_\C \yoneda{c}\). But this is identical to \cite[Section
4]{swanwtypesreductions} so we omit the details.

If we then define \(W'(Y, X, R)(c, \gamma)\) to be \(N(c, 1_c)\), then
this is strictly stable under reindexing by definition.

One can construct by recursion an isomorphism between \(N(d, f)\) and
\(W(Y, X, R)(d, \Gamma(f)(\gamma))\) for each \(f : d \to c\). In
particular this gives us an isomorphism between \(N(c, 1_c)\) and
\(W(Y, X, R)(c, \gamma)\), and so we have a canonical isomorphism
between \(W'(Y, X, R)(c, \gamma)\) and \(W(Y, X, R)(c, \gamma)\). It
follows that we can assign an initial algebra structure to
\(W'(Y, X, R)(c, \gamma)\) by transferring the algebra structure on
\(W(Y, X, R)(c, \gamma)\) via the isomorphism.

\endgroup

\endgroup

\subsection{Discrete Types}
\label{sec:discrete-types}

We introduce a class of types in an Orton-Pitts model for future
use. Let \(\cat{E}\) be a model of type theory satisfying Assumption
\ref{asm:orton-pitts}.

\begin{definition}
  A type \(1 \vdash X\) is said to be \emph{discrete} if the
  map \(\lambda x.\lambda i.x : X \to X^{\I}\) is an isomorphism.
\end{definition}

The proofs of the following propositions are found in
\cite{cubicalassemblies}.

\begin{proposition}
  \label{prop:discrete-composition}
  Every discrete type \(1 \vdash X\) carries a fibration
  structure.
\end{proposition}

\begin{proposition}
  \label{prop:decidable-discrete}
  If a type \(1 \vdash X\) has decidable equality, then it is
  discrete.
\end{proposition}

\begin{corollary}
  \label{cor:nno-discrete}
  The natural number object in \(\cat{E}\) is discrete.
\end{corollary}

\section{Church's Thesis}
\label{sec:churchs-thesis}

We consider a dependent type theory with dependent product types,
dependent sum types, identity types, unit type, disjoint finite
coproducts, propositional truncation and natural numbers. In such a
dependent type theory, one can define Kleene's computation predicate
\(T(e, x, z)\) and result extraction function \(U(z)\) as primitive
recursive functions \(T : \N \times \N \times \N \to \Bool\) and \(U : \N \to
\N\). The statement \(T(e, x, z)\) means that \(z\) codes a
computation on Turing machine \(e\) with input \(x\) and \(U(z)\) is
the output of the computation. \emph{Church's Thesis} is the following
axiom.
\[
  \forall_{f : \N \to \N}\exists_{e : \N}\forall_{x : \N}\exists_{z : \N}T(e, x, z) \land U(z)
  = f(x)
\]
Since the type \(\sum_{z : \N}T(e, x, z) \times U(z) = f(x)\) is a
proposition, Church's Thesis is equivalent to the type
\[
  \prod_{f : \N \to \N}\trunc{\sum_{e : \N}\prod_{x : \N}\sum_{z : \N}T(e, x, z) \times U(z) =
  f(x)}.
\]

\subsection{Failure of Church's Thesis in Internal Cubical Models}
\label{sec:failure-in-presheaves}
\begingroup
\newcommand{\E}{\mathcal{S}}
\newcommand{\C}{\Box}
\newcommand{\yoneda}{\mathbf{y}}
\newcommand{\adj}{\dashv}       
\newcommand{\vadj}{\rotatebox[origin=c]{270}{\(\adj\)}} 

Let \(\E\) be a model of type theory as in Section
\ref{sec:presheaf-models}. We have seen that the category \(\PSh(\C)\)
of internal cubical objects in \(\E\) is part of a model of type
theory satisfying Assumption \ref{asm:orton-pitts}. In this section we
show the following theorem.

\begin{theorem}
  \label{thm:churchs-thesis-fails-in-presheaves}
  The negation of Church's Thesis holds in the model of univalent type
  theory \(\cttmodel{\PSh(\C)}\).
\end{theorem}

To prove Theorem \ref{thm:churchs-thesis-fails-in-presheaves}, we
recall from \cite{cubicalassemblies} the notion of a codiscrete
presheaf. The constant presheaf functor \(\Delta : \E \to \PSh(\C)\)
extends to a morphism of cwf's and preserves (at least up to
isomorphism) several type constructors. Here we only need the
following.

\begin{proposition}
  \label{prop:constant-presheaf-preservation}
  The morphism \(\Delta : \E \to \PSh(\C)\) of cwf's preserves
  dependent product types, dependent sum types, extensional identity
  types and natural number objects.
\end{proposition}

A constant presheaf \(\Delta X\) is regarded as a type in
\(\cttmodel{\PSh(\C)}\) by the following proposition and Proposition
\ref{prop:discrete-composition}.

\begin{proposition}
  \label{prop:constant-presheaf-discrete}
  Constant presheaves are discrete.
\end{proposition}

For types \(1 \vdash_{\E} X\) and \(x : X \vdash_{\E} Y(x)\), one can
define a type \(x : \Delta X \vdash_{\PSh(\C)} \nabla_{X}Y(x)\) called
the codiscrete presheaf which has the following properties.

\begin{proposition}
  \(\nabla_{X}\) is the right adjoint to the evaluation functor
  \((-)_{0}\) at \(0 \in \C\): for any type
  \(x : \Delta X \vdash_{\PSh(\C)} Z(x)\), we have a natural bijection
  between the set of elements
  \(x : \Delta X, z : Z(x) \vdash_{\PSh(\C)} b : \nabla_{X}Y(x)\) and
  the set of elements \(x : X, z : Z_{0}(x) \vdash_{\E} b :
  Y(x)\). Note that \((\Delta X)_{0} = X\) and thus \(Z_{0}\) is a
  type in \(\E\) over \(X\).
\end{proposition}

\begin{proposition}
  \label{prop:codiscrete-proposition}
  For a type \(x : X \vdash_{\E}Y(x)\), the type
  \(x : \Delta X \vdash_{\PSh(\C)} \nabla_{X}Y(x)\) has a
  composition structure and is a proposition in
  \(\cttmodel{\PSh(\C)}\).
\end{proposition}

\begin{proof}[Proof of Theorem \ref{thm:churchs-thesis-fails-in-presheaves}]
  We define types
  \(f : \N \to \N \vdash C'(f) :\equiv \sum_{e : \N}\prod_{x :
    \N}\sum_{z : \N}T(e, x, z) \times U(x) = f(x)\) and
  \(f : \N \to \N \vdash C(f) :\equiv \trunc{C'(f)}\). Let \(N\)
  denote the natural number object in \(\E\). Then Church's Thesis is
  interpreted in \(\cttmodel{\PSh(\C)}\) as
  \(\prod_{f : \Delta(N \to N)}\itpr{C}^{\cttmodel{\PSh(\C)}}(f)\) by
  Proposition \ref{prop:constant-presheaf-preservation}. We will
  construct two functions in \(\cttmodel{\PSh(\C)}\):
  \begin{itemize}
  \item
    \(\prod_{f : \Delta(N \to N)}\itpr{C}^{\cttmodel{\PSh(\C)}}(f) \to
    \nabla_{N \to N}\itpr{C'}^{\E}(f)\);
  \item
    \(\left(\prod_{f : \Delta(N \to N)}\nabla_{N \to
        N}\itpr{C'}^{\E}(f)\right) \to \Initial\).
  \end{itemize}
  Then we readily get a function
  \(\left(\prod_{f : \Delta(N \to
      N)}\itpr{C}^{\cttmodel{\PSh(\C)}}(f)\right) \to \Initial\).

  For the former one it suffices to give a function
  \(\itpr{C'}^{\cttmodel{\PSh(\C)}}(f) \to \nabla_{N \to
    N}\itpr{C'}^{\E}(f)\) for all \(f : \Delta(N \to N)\) by the
  recursion principle of the propositional truncation because the
  codomain is a proposition by Proposition
  \ref{prop:codiscrete-proposition}. By the adjunction
  \((-)_{0} \adj \nabla_{N \to N}\) it suffices to give a function
  \(\itpr{C'}^{\cttmodel{\PSh(\C)}}_{0} \to \itpr{C'}^{\E}\) but we
  have an isomorphism
  \(\itpr{C'}^{\cttmodel{\PSh(\C)}} \cong (\Delta\itpr{C'}^{\E})_{0} =
  \itpr{C'}^{\E}\) by Proposition
  \ref{prop:constant-presheaf-preservation}.

  For the latter function, observe that
  \(\prod_{f : \Delta(N \to N)}\nabla_{N \to N}\itpr{C'}^{\E}(f) \cong
  \nabla_{1}\left(\prod_{f : N \to N}\itpr{C'}^{\E}(f)\right)\) and
  that \(\nabla_{1}\Initial \cong \Initial\). Then we apply
  \(\nabla_{1}\) to the function
  \(\left(\prod_{f : N \to N}\itpr{C'}^{\E}(f)\right) \to \Initial\)
  in \(\E\) obtained from the inconsistency of Church's Thesis with
  the axiom of choice and function extensionality.
\end{proof}
\endgroup

\section{Null Types}
\label{sec:modalities}
\begingroup
\newcommand{\E}{\mathcal{E}}
\newcommand{\action}{\cdot}
\newcommand{\isequiv}{\mathsf{isEquiv}}
\newcommand{\isnull}{\mathsf{isNull}}

Let \(\E\) be a model of univalent type theory. Based on Rijke, Shulman
and Spitters' null types \cite{rijkeshulmanspitters} we define a
notion of null structure as follows.

Let \(a : A \vdash
B(a)\) be a proposition in \(\E\). For a type \(\Gamma \vdash X\) in
\(\E\), we define a proposition \(\Gamma \vdash \isnull_{B}(X)\) as
\[
  \Gamma \vdash \prod_{a : A}\isequiv(\lambda(x : X).\lambda(b : B(a)).x)
\]
and call a term of \(\isnull_{B}(X)\) a \emph{\(B\)-null structure on
  \(X\)}.
A \emph{\(B\)-null type} is a type \(\Gamma \vdash X\) equipped with a
\(B\)-null structure \(n\) on \(X\). That is, a \(B\)-null type has a
witness that the canonical map \(X \to X^{B(a)}\) is an equivalence
for each \(a\).

\begin{definition}
  We define a cwf \(\E_{B}\) as follows:
  \begin{itemize}
  \item the contexts are those of \(\E\);
  \item the types are the \(B\)-null types in \(\E\);
  \item the elements of \(\Gamma \vdash_{\E_{B}} X\) are those of the
    underlying type \(X\) in \(\E\).
  \end{itemize}
  We have the obvious forgetful morphism \(\E_{B} \to \E\) of cwf's.
\end{definition}

For a proposition \(a : A \vdash B(a)\), a \emph{nullification
  operator} assigns
\begin{itemize}
\item each type \(\Gamma \vdash X\) a \(B\)-null type
  \(\Gamma \vdash \loc_{B}X\) and an element
  \(\Gamma \vdash \eta_{X} : X \to \loc_{B}X\); and
\item each pair of type \(\Gamma \vdash X\) and \(B\)-null type
  \(\Gamma \vdash Y\) an element
  \(\Gamma \vdash e : \isequiv(\lambda(f : \loc_{B}X \to Y).f \circ
  \eta_{X}\).
\end{itemize}
We also require that a nullification operator is preserved by
reindexing.

We review some properties of null types. See
\cite{rijkeshulmanspitters} for further details.

\begin{proposition}
  \label{prop:null-closure}
  Let \(\Gamma \vdash X\) and \(\Gamma, x : X \vdash Y(x)\) be types
  in \(\E\).
  \begin{itemize}
    \item There exists a term of type \(\Gamma \vdash (\prod_{x :
        X}\isnull_{B}(Y(x))) \to \isnull_{B}(\prod_{x : X}Y(x))\).
    \item There exists a term of type \(\Gamma \vdash \isnull_{B}(X)
      \to (\prod_{x : X}\isnull_{B}(Y(x))) \to \isnull_{B}(\sum_{x :
        X}Y(x))\).
    \item There exists a term of type \(\Gamma \vdash \isnull_{B}(X)
      \to \prod_{x_{0}, x_{1} : X}\isnull_{B}(\Id_{X}(x_{0},
      x_{1}))\).
  \end{itemize}
  Consequently, \(\E_{B}\) supports dependent product, dependent sum
  and intensional identity types preserved by the morphism
  \(\E_{B} \to \E\).
\end{proposition}

For a universe \(U\) we define a subuniverse \(U_{B}\) of \(U\) as
\[
  U_{B} \equiv \{X : U \mid \isnull_{B}(X)\}.
\]

\begin{proposition}
  \label{prop:nullification-universe}
  The universe \(U_{B}\) has a \(B\)-null structure.
\end{proposition}

\begin{proof}
  Our condition that each \(B(a)\) is a proposition corresponds to
  Rijke, Shulman and Spitters' notion of \emph{topological
    modality}. They prove in \cite[Corollary 3.11 and Theorem
  3.12]{rijkeshulmanspitters} that for any such modality the universe
  of modal types is itself modal.
\end{proof}

\begin{proposition}
  \label{prop:nullification-preserves-prop}
  If a nullification operator \(\loc_{B}\) exists, then it preserves
  propositions.
\end{proposition}

\begin{proof}
  This is true for any modality by \cite[Lemma
  1.28]{rijkeshulmanspitters}.
\end{proof}

\begin{corollary}
  \label{cor:nullification-contractible}
  Suppose that \(\E\) has a nullification operator \(\loc_{B}\). Then
  \(a : A \vdash \loc_{B}B(a)\) is contractible.
\end{corollary}
\begin{proof}
  Since \(\loc_{B}B(a)\) is a proposition, it suffices to find an
  element of \(\prod_{a : A}\loc_{B}B(a)\). Assume that \(a : A\) is
  given. Since \(\loc_{B}B(a)\) is \(B\)-null, it is enough to give a
  function \(B(a) \to \loc_{B}B(a)\), so take the constructor
  \(\eta_{B(a)} : B(a) \to \loc_{B}B(a)\).
\end{proof}

\begin{corollary}
  \label{cor:truncation-of-null-types}
  Suppose that \(\E\) has a nullification operator \(\loc_{B}\). Then
  \(X \mapsto \loc_{B}\trunc{X}\) gives propositional truncation in
  the model \(\E_{B}\).
\end{corollary}
\begin{proof}
  By Proposition \ref{prop:nullification-preserves-prop},
  \(\loc_{B}\trunc{X}\) is a proposition. For any \(B\)-null proposition
  \(Z\), we have equivalences
  \begin{align*}
    (\loc_{B}\trunc{X} \to Z) &\simeq (\trunc{X} \to Z) \\
                          &\simeq (X \to Z).
  \end{align*}
\end{proof}

\subsection{Null Types in Orton-Pitts Models}
\label{sec:null-types-in-orton-pitts}
\begingroup

Let \(\cttmodel{\E}\) be an Orton-Pitts model.

\begin{definition}
  A type \(a : A \vdash B(a)\) in \(\E\) or \(\cttmodel{\E}\) is said
  to be \emph{well-supported} if the propositional truncation
  \(a : A \vdash_{\E} \trunc{B(a)}\) taken in the model \(\E\) of
  extensional dependent type theory is inhabited.
\end{definition}

\begin{proposition}
  \label{prop:discrete-null}
  Let \(1 \vdash_{\E} X\) be a type and
  \(a : A \vdash_{\cttmodel{\E}} B(a)\) a proposition. If \(X\) is
  discrete and \(B\) is well-supported, then \(X\) has a \(B\)-null
  structure.
\end{proposition}
\begin{proof}
  We show that, for any \(a : A\), the function \(k_{a} :\equiv
  \lambda x.\lambda b.x : X \to (B(a) \to X)\) is an isomorphism in
  the internal language of \(\E\). Since \(B\) is well-supported,
  \(k_{a}\) is injective. To prove surjectivity we assume that \(f
  : B(a) \to X\) is given. By the well-supportedness of \(B\) there
  exists some element \(b : B(a)\). We show that \(f =
  k_{a}(f(b))\). Assume \(b' : B(a)\) is given. Since \(B\) is a
  proposition in \(\cttmodel{E}\) we have a path \(p : \I \to B(a)\)
  such that \(p\Izero = b'\) and \(p\Ione = b\). By the discreteness
  of \(X\) the path \(f \circ p : \I \to X\) is constant, which
  implies that \(f(b') = f(b)\). Hence we have \(f = k_{a}(f(b))\) by
  function extensionality.
\end{proof}

We easily deduce the following corollaries.
\begin{corollary}
  \label{cor:botnull}
  If \(B\) is well-supported, then \(\Initial\) has a \(B\)-null
  structure.
\end{corollary}

\begin{corollary}
  \label{cor:natnull}
  If \(B\) is well-supported, then \(\N\) has a \(B\)-null structure.
\end{corollary}

\begin{proposition}
  \label{prop:coproduct-null}
  Let \(a : A \vdash_{\cttmodel{\E}} B(a)\) be a proposition and
  \(\Gamma \vdash_{\cttmodel{\E}} X\) and
  \(\Gamma \vdash_{\cttmodel{\E}} Y\) types. Then there exists a term
  of type
  \[
    \Gamma \vdash \isnull_{B}(X) \to \isnull_{B}(Y) \to \isnull_{B}(X
    + Y).
  \]
\end{proposition}
\begin{proof}
  We proceed in the internal language of \(\E\). Suppose that \(X\)
  and \(Y\) has a \(B\)-null structure. Assume that \(a : A\) is
  given. Since the function \((X + Y) \to (B(a) \to (X + Y))\) factors
  as
  \[
    \begin{tikzcd}
      X + Y \arrow[rr] \arrow[dr] & & B(a) \to (X + Y) \\
      & (B(a) \to X) + (B(a) \to Y), \arrow[ur,"\Phi"'] &
    \end{tikzcd}
  \]
  it suffices to show that the function
  \(\Phi : ((B(a) \to X) + (B(a) \to Y)) \to (B(a) \to (X + Y))\) is
  an isomorphism. The injectivity of \(\Phi\) follows from the
  well-supportedness of \(B\). To prove the surjectivity we assume
  that \(f : B(a) \to (X + Y)\) is given. We show that
  \((\forall_{b : B(a)}fb \in X) \lor (\forall_{b : B(a)}fb \in
  Y)\). Since \(B\) is well-supported, there exists some element
  \(b_{0} : B(a)\). We know that \(fb_{0} \in X \lor fb_{0} \in
  Y\). Suppose that \(fb_{0} \in X\). Assume \(b : B(a)\) is
  given. Since \(B\) is a proposition in \(\cttmodel{\E}\), we have a
  path \(p : \I \to B(a)\) such that \(p0 = b_{0}\) and \(p1 =
  b\). Since the exponential functor \((-)^{\I}\) preserves colimits
  because it has a right adjoint, we have
  \((\forall_{i : \I}f(pi) \in X) \lor (\forall_{i : \I}f(pi) \in
  Y)\). Now \(f(p0) \in X\) and thus we have
  \(\forall_{i : \I}f(pi) \in X\). In particular, \(fb \in X\). In a
  similar manner, we have \(\forall_{b : B(a)}fb \in Y\) assuming
  \(fb_{0} \in Y\). Hence we get
  \((\forall_{b : B(a)}fb \in X) \lor (\forall_{b : B(a)}fb \in Y)\).
\end{proof}

By Corollary \ref{cor:nno-discrete}, Propositions
\ref{prop:null-closure}, \ref{prop:discrete-null} and
\ref{prop:coproduct-null}, for any type \(X\) defined in dependent
type theory only using dependent product types, dependent sum types,
identity types, unit type, disjoint finite coproducts and natural
numbers, the interpretation \(\itpr{X}^{\cttmodel{\E}}\) has a
\(B\)-null structure for any well-supported proposition \(B\) in
\(\cttmodel{\E}\). In particular, if \(\itpr{X}^{\cttmodel{\E}}\) is
inhabited, then so is \(\itpr{X}^{\cttmodel{\E}_{B}}\) for any
well-supported proposition \(B\) in \(\cttmodel{\E}\).

\begin{example}
  \label{ex:mpinnulltypes}
  \emph{Markov's Principle} is the following axiom.
  \[
    \forall_{\alpha : \N \to \Bool}\neg\neg(\exists_{n : \N}\alpha(n))
    \to \exists_{n : \N}\alpha(n)
  \]
  It is equivalent to the type
  \[
    \prod_{\alpha : \N \to \Bool}\prod_{p : (\prod_{n : \N}\alpha(n)
      \to \Initial) \to \Initial}\trunc{\sum_{n : \N}\alpha(n)}.
  \]
  For a decidable predicate \(\alpha : \N \to \Bool\), the proposition
  \(\trunc{\sum_{n : \N}\alpha(n)}\) is equivalent to the type
  \(\sum_{n : \N}\alpha(n) \times \prod_{k : \N}\alpha(k) \to n \leq
  k\) which is defined without propositional truncation. Hence, if the
  model \(\E\) of extensional dependent type theory satisfies Markov's
  Principle, then so does the model \(\cttmodel{\E}_{B}\) of univalent
  type theory for any well-supported proposition \(B\) in
  \(\cttmodel{\E}\).
\end{example}

We now show how to define nullification operators in Orton-Pitts
models. Following Rijke, Shulman and Spitters in \cite[Section
2.2]{rijkeshulmanspitters} we will first define an operator
\(\Jop_B\), although we will only consider the case of nullification,
since that is all we need here.

\begin{lemma}
  \label{lem:Joperatorexists}
  For types \(a : A \vdash_{\cttmodel{\E}} B(a)\) and
  \(\Gamma \vdash_{\cttmodel{\E}} X\), we have the higher inductive
  \(\Jop_B(X)\) defined as follows.
  \begin{itemize}
  \item When \(x : X\), then \(\Jop_B(X)\) contains an element
    \(\alpha^B_X(x)\).
  \item When \(a : A\) and \(f : B(a) \to \Kop_B\), then \(\Jop_B(X)\)
    contains an element \(\ext(a, f)\).
  \item When \(a : A\), \(f : B(a) \to \Kop_B\) and \(b : B(a)\) then
    \(\Id(\ext(a, f), f(b))\) contains an element \(\isext(a, f, b)\).
  \end{itemize}
\end{lemma}

\begin{proof}
  \(\Jop_B(X)\) differs from \(\Kop_B\) by having an extra point
  constructor \(\alpha^B_X : X \to \Jop_B(X)\).

  We define \(A'\) to be the type \(A + X\) and define the family of
  types \(a : A' \vdash B'(a)\) as follows.
  \begin{align*}
    B'(\inl(a)) &:\equiv B(a) \\
    B'(\inr(x)) &:\equiv 0
  \end{align*}
  We can then take \(\Jop_B(X)\) to be \(\Kop_{B'}\), as defined in
  section \ref{sec:high-induct-types}. We take \(\alpha^B_X(x)\) to be
  \(\ext(\inr(x), \bot_{\Kop_B})\) where \(\bot_{\Kop_B}\) is the
  unique map from \(0\) to \(\Kop_B\).
\end{proof}

\begin{theorem}
  \(\cttmodel{\E}\) has a nullification operator \(\loc_{B}\) for
  every type \(a : A \vdash_{\cttmodel{\E}} B(a)\).
\end{theorem}

\begin{proof}
  This follows from \cite[Theorem 2.16]{rijkeshulmanspitters},
  observing that for the case of nullification the pushout appearing
  there is just a suspension, which we have already shown how to
  implement in theorem \ref{thm:suspensionexist}, and we
  showed in lemma \ref{lem:Joperatorexists} how to implement their
  \(\Jop\) operator.
\end{proof}
\endgroup
\endgroup

\section{Church's Thesis in Null Types}
\label{sec:churchs-thesis-in-null-types}
\begingroup
\newcommand{\E}{\mathcal{E}}
\newcommand{\cS}{\mathcal{S}}
\newcommand{\KOne}{\mathcal{K}_{1}}

Consider a dependent type theory with dependent product
types, dependent sum types, identity types, unit type, disjoint finite
coproducts, propositional truncation and natural numbers. Let \(a : A
\vdash B(a)\) be a type in this type theory where \(A\) and \(B\) are
definable only using
dependent product types, dependent sum types, identity type of
\(\Bool\), unit type, finite coproducts and natural numbers. We define
\(a : A \vdash C(a) := \trunc{B(a)}\). For an Orton-Pitts model
\(\cttmodel{\E}\), the underlying types of the interpretations
\(\itpr{A}^{\cttmodel{\E}}\) and \(\itpr{B}^{\cttmodel{\E}}\) are
\(\itpr{A}^{\E}\) and \(\itpr{B}^{\E}\) respectively.

\begin{theorem}
  \label{thm:unnamed1}
  Let \(\cttmodel{\E}\) be an Orton-Pitts model. Suppose that
  \(\itpr{B}^{\cttmodel{\E}}\) is well-supported. Then the proposition
  \(C\) holds in the model of univalent type theory
  \(\cttmodel{\E}_{\itpr{C}^{\cttmodel{\E}}}\).
\end{theorem}
\begin{proof}
  Let \(D = \itpr{C}^{\cttmodel{\E}}\). By assumption
  \(\itpr{B}^{\cttmodel{\E}}\) is well-supported and so is its
  truncation \(D\). By Corollary \ref{cor:nno-discrete} and
  Propositions \ref{prop:null-closure}, \ref{prop:discrete-null} and
  \ref{prop:coproduct-null}, \(\itpr{A}^{\cttmodel{\E}}\) and
  \(\itpr{B}^{\cttmodel{\E}}\) have \(D\)-null structures. Hence we
  have \(\itpr{C}^{\cttmodel{\E}_{D}} =
  \loc_{D}\itpr{C}^{\cttmodel{\E}}\) by Corollary
  \ref{cor:truncation-of-null-types}, and this type is inhabited by
  Corollary \ref{cor:nullification-contractible}.
\end{proof}

\begin{corollary}
  \label{cor:unnamed2}
  Let \(\cS\) be a model of type theory as in Section
  \ref{sec:presheaf-models}. If the proposition \(C\) holds in
  \(\cS\), then \(C\) also holds in the model of univalent type theory
  \(\cttmodel{\E}_{\itpr{C}^{\cttmodel{\E}}}\) where \(\E =
  \PSh(\Box)\).
\end{corollary}
\begin{proof}
  Since \(C = \trunc{B}\) holds in \(\cS\), the type
  \(\itpr{B}^{\cS}\) is well-supported. Then
  \(\itpr{B}^{\cttmodel{\E}} = \itpr{B}^{\E}\) is also well-supported
  because the constant presheaf functor \(\cS \to \E\) preserves all
  structures of the type theory. Then use Theorem \ref{thm:unnamed1}.
\end{proof}

\begin{example}
  \label{exm:church-thesis}
  Recall that Church's Thesis is equivalent to the type
  \[
    \prod_{f : \N \to \N}\trunc{\sum_{e : \N}\prod_{x : \N}\sum_{z :
        \N}T(e, x, z) \times U(z) = f(x)}.
  \]
  Also note that the equality of natural numbers is decidable, and
  thus there exists a function \(=_{\N} : \N \to \N \to \Bool\) such
  that the type \(U(z) = f(x)\) is equivalent to \((U(z) =_{\N} f(x))
  = 1\). Therefore Church's Thesis is equivalent to a type of the form
  \[
    \prod_{a : A}\trunc{B(a)}
  \]
  with a type \(a : A \vdash B(a)\) definable only using dependent
  product types, dependent sum types, identity of \(\Bool\), unit
  type, finite coproducts and natural numbers. Since Church's Thesis
  holds in the category \(\Asm(\KOne)\) of assemblies on Kleene's
  first model \(\KOne\), by Corollary \ref{cor:unnamed2} the model of
  univalent type theory \(\cttmodel{\E}_{\itpr{C}^{\cttmodel{\E}}}\)
  satisfies Church's Thesis where \(\E = \CAsm(\KOne)\).
\end{example}

We can now prove our second main result, which informally says that
univalent type theory is consistent with the main principles of
Recursive Constructive Mathematics.
\begin{theorem}
  \label{thm:ctconsistent}
  Martin-L\"{o}f type theory remains consistent when all of the
  following extra structure and axioms are added.
  \begin{enumerate}
  \item Propositional truncation.
  \item The axiom of univalence.
  \item Church's Thesis.
  \item Markov's Principle.      
  \end{enumerate}
\end{theorem}

\begin{proof}
  We prove consistency by constructing a model where all of the above
  holds and where there is no element of type \(\bot\). Consider the
  Orton-Pitts model \(\cttmodel{\E}\) with \(\E = \CAsm(\KOne)\). We
  have seen that \(\cttmodel{\E}\) satisfies Church's Thesis in
  Example \ref{exm:church-thesis}. It remains to show that
  \(\cttmodel{\E}\) satisfies Markov's Principle and \(\bot\) is empty
  in this model.

  Using well supportness again, and example \ref{ex:mpinnulltypes} we
  see that to show Markov's principle holds, it suffices to show
  it holds in cubical assemblies (as a model of extensional type
  theory). Again, we observe that the type corresponding to Markov's
  principle is preserved by the constant presheaves functor, and so it
  suffices to show that Markov's principle holds in assemblies, which
  is again a standard argument.

  Using well supportness once more, and corollary \ref{cor:botnull} we
  see that \(\bot\) is the same in null types as in cubical
  assemblies. It follows that it has no global sections, i.e. there
  is no element of type \(\bot\) in the model.
\end{proof}

We can use Theorem \ref{thm:unnamed1} for other principles.

\begin{example}
  \emph{Brouwer's Continuity Principle} is the following axiom.
  \[
    \forall_{F : (\N \to \N) \to \N}\forall_{\alpha : \N \to
      \N}\exists_{n : \N}\forall_{\beta : \N \to \N}(\forall_{m : \N}m
    < n \to \alpha(m) = \beta(m)) \to F(\alpha) = F(\beta)
  \]
  The standard ordering \(<\) on \(\N\) is decidable, and thus
  Brouwer's Continuity Principle is an instance of Theorem
  \ref{thm:unnamed1}.
\end{example}

We obtain a new proof of the following result originally proved by
Coquand using cubical stacks
\cite{coquandcubicalstacks}\footnote{Coquand's proof can also be said
  to use reflective subuniverses, although in a very different way to
  our proof.}. See also \cite{cmrstacks} for an earlier stack model
based on groupoids.

\begin{theorem}
  Martin-L\"{o}f type theory remains consistent when all of the
  following extra structure and axioms are added.
  \begin{enumerate}
  \item Propositional truncation.
  \item The axiom of univalence.
  \item Brouwer's Continuity Principle.
  \end{enumerate}  
\end{theorem}

\begin{proof}
  This is the same as for theorem \ref{thm:ctconsistent}. See
  e.g. \cite[Proposition 3.1.6]{vanoosten} for a proof that Brouwer's
  principle holds in the the effective topos (the same proof applies
  for assemblies).
\end{proof}

\endgroup

\section{Conclusion and Further Work}
\label{sec:concl-furth-work}

We have constructed a model of type theory that satisfies the main
axiom of homotopy type theory (univalence) and the main axioms of
recursive constructive mathematics (Church's thesis and Markov's
principle). However, in both fields there are additional axioms that
are natural to consider, but which we have left for future work.

With regards to homotopy type theory, we expect that the
remaining higher inductive types appearing in \cite{hottbook} can be
implemented following the technique suggested in \cite[Remark
3.23]{rijkeshulmanspitters} together with the technique of
\cite{coquandhubermortberg} for constructing the necessary higher
inductive types in cubical assemblies.

The situation with the remaining axioms of recursive constructive
mathematics is more difficult. The axiom of countable choice is often
included, but it is unclear whether countable choice holds in our
model, or how to adjust the model to ensure countable choice does
hold. The other main axiom of recursive constructive mathematics is
extended Church's thesis, which states that certain partial functions
from \(\N\) to \(\N\) are computable. The main issue here is that it
is unclear what is the most natural way to formulate partial functions
in homotopy type theory. Much progress on this has been made by
Escard\'{o} and Knapp in \cite{escardoknapp}. However, as they show, a
weak form of countable choice is needed for their definition to work
as expected. We expect that for any reasonable formulation of extended
Church's thesis Theorem \ref{thm:unnamed1} can be used to construct
a model where it holds.

Another open problem is to find a good definition of
\((\infty, 1)\)-effective topos, which should be to the effective
topos what \((\infty, 1)\)-toposes are to Grothendieck toposes. In
particular the effective topos should be recovered as the localisation
of the hsets in the \((\infty, 1)\)-effective topos, and commonly seen
theorems and definitions in the effective topos should be special
cases of corresponding higher versions. One possible definition is
cubical assemblies. We can now see another possibility in the form of
reflective subuniverses of cubical assemblies. However, our definition
is dependent on particular a choice of axioms that satisfy the
necessary conditions to apply Theorem \ref{thm:unnamed1}, so we leave
open the problem of finding a ``natural'' definition that satisfies
axioms such as Church's thesis without needing to ensure they hold in
the definition.

\bibliography{references}

\end{document}